\documentclass[a4paper]{amsart}

\usepackage{latexsym, amssymb,amsmath, amsthm,amsfonts}

\newcommand{\C}{\mathbb{C}}
\newcommand{\kk}{\Bbbk}

\def\wt{\operatorname{wt}}

\def\sgn{\operatorname{sgn}}

\def\supp{\operatorname{Supp}}

\def\N{\mathbb{N}}

\def\im{\operatorname{im}}

\def\codim{\operatorname{codim}}

\def\Tr{\operatorname{Tr}}

\newtheorem{Lemma}{Lemma}
\newtheorem{Theorem}[Lemma]{Theorem}

\newtheorem{Corollary}[Lemma]{Corollary}
\newtheorem{Prop}[Lemma]{Proposition}

\newtheorem*{Corollary of Conjecture}{Corollary of Conjecture}

\theoremstyle{definition}

\theoremstyle{remark}
  \newtheorem{rem}[Lemma]{Remark}

\newtheoremstyle{Acknowledgments}
  {}
    {}
     {}
     {}
    {\bfseries}
    {}
     {.5em}
     {\thmname{#1}\thmnumber{ }\thmnote{ (#3)}}
\theoremstyle{Acknowledgments}

\title{Modular Covariants of Cyclic Groups of order $p$}

\author{Jonathan Elmer}
\address{Middlesex University\\
The Burroughs, London\\
NW4 4BT}
\email{j.elmer@mdx.ac.uk}

\date{\today}

\begin{document}

\maketitle
\begin{abstract} Let $G$ be a cyclic group of order $p$ and let $V, W$ be $\kk G$-modules. We study the modules of covariants $\kk[V,W]^G = (S(V^*) \otimes W)^G$. For $V$ indecomposable with dimension 2, and $W$ an arbitrary indecomposable module, we show $\kk[V,W]^G$ is a free $\kk[V]^G$-module (recovering a result of Broer and Chuai \cite{BroerChuaiRelative}) and we give an explicit set of covariants generating $\kk[V,W]^G$ freely over $\kk[V]^G$. For $V$ indecomposable with dimension 3, and $W$ an arbitrary indecomposable module, we show that $\kk[V,W]^G$ is a Cohen-Macaulay $\kk[V]^G$-module (again recovering a result of Broer and Chuai) and we give an explicit set of covariants which generate $\kk[V,W]^G$ freely over a homogeneous system of parameters for $\kk[V]^G$. We also use our results to compute a minimal generating set for the transfer ideal of $\kk[V]^G$ over a homogeneous sytem of parameters when $V$ has dimension 3.
\end{abstract}

\section{Introduction}
Let $G$ be a group, $\kk$ a field, and $V$ and $W$ finite-dimensional $\kk G$-modules on which $G$ acts linearly. Then $G$ acts on the set of functions $V \rightarrow W$ according to the formula

\[g \cdot \phi(v) = g \phi(g^{-1} v)\] for all $g \in G$ and $v \in V$.

Classically, a \emph{covariant} is a $G$-equivariant polynomial map $V \rightarrow W$. An \emph{invariant} is the name given to a covariant $V \rightarrow \kk$ where $\kk$ denotes the trivial indecomposable $\kk G$-module. If the field $\kk$ is infinite, then the set of polynomial maps $V \rightarrow W$ can be identified with $S(V^*) \otimes W$, where the action on the tensor product is diagonal and the action on $S(V^*)$ is the natural extension of the action on $V^*$ by algebra automorphisms. Then the natural pairing $S(V^*) \times S(V^*) \rightarrow S(V^*)$ is compatible with the action of $G$, and makes the invariants $S(V^*)^G$ a $\kk$-algebra, and the covariants $(S(V^*) \otimes W)^G$ a $S(V^*)^G$-module.

If $G$ is finite and the characteristic of $\kk$ does not divide $|G|$, then Schur's lemma implies that every covariant restricts to an isomorphism of some direct summand of $S(V^*)$ onto $W$. Thus, covariants can be viewed as ``copies'' of $W$ inside $S(V^*)$. Otherwise, the situation is more complicated.

The algebra of polynomial maps $V \rightarrow \kk$ is usually written as $\kk[V]$. In this article we will write $\kk[V]^G$ for the algebra of $G$-invariants, and $\kk[V,W]^G$ for the module of covariants. We are interested in the structure of $\kk[V,W]^G$ as a $\kk[V]^G$-module. Throughout, $G$ denotes a finite group.

This question has been considered by a number of authors over the years. For example, Chevalley and Sheppard-Todd \cite{Chevalley}, \cite{SheppardTodd} showed that if the characteristic of $\kk$ does not divide $|G|$ and $G$ acts as a reflection group on $V$, then $\kk[V]^G$ is a polynomial algebra and $\kk[V,W]^G$ is free. More generally, Eagon and Hochster \cite{HochsterEagon} showed that if the characteristic of $\kk$ does not divide $|G|$ then $\kk[V,W]^G$ is a Cohen-Macaulay module (and $\kk[V]^G$ a Cohen-Macaulay ring in particular). In the modular case, Hartmann \cite{Hartmann} and Hartmann-Shepler \cite{HartmannShepler} gave necessary and sufficient conditions for a set of covariants to generate $\kk[V,W]^G$ as a free $\kk[V]^G$-module, provided that $\kk[V]^G$ is polynomial and $W \cong V^*$. Broer and Chuai \cite{BroerChuaiRelative} remove the restrictions on both $W$ and $\kk[V]^G$.

The present article is inspired by two particular results from \cite{BroerChuaiRelative}, which we state here for convenience:

\begin{Prop}[\cite{BroerChuaiRelative}, Proposition 6]\label{broerchuai}  Let $G$ be a finite group of order divisible by $p =$ char$(\kk)$ and let $V,W$ be $\kk G$-modules. 
\begin{enumerate}
\item[(i)] Suppose $\codim(V^G) = 1$. Then $\kk[V]^G$ is a polynomial algebra and $\kk[V,W]^G$ is free as a graded module over $\kk[V]^G$.

\item[(ii)] Suppose $\codim(V^G) = 2$. Then $\kk[V,W]^G$ is a Cohen-Macaulay graded module
over $\kk[V]^G$ .
\end{enumerate}
\end{Prop}

In the situation of (i) above, there is a method for checking a set of covariants generates $\kk[V,W]^G$ over $\kk[V]^G$, but no method of constructing generators. Meanwhile, in the situation of (ii), there exists a polynomial subalgebra $A$ of $\kk[V]^G$ over which $\kk[V,W]^G$ is a free module. It is not clear how to find module generators, or to check that they generate $\kk[V,W]^G$. 

The purpose of this article is to work towards making these results constructive. We investigate certain modules of covariants for $V$ satisfying (i) or (ii) above and $G$ a cyclic group of order $p$.

\section{Preliminaries}\label{sec:prelim}

From this point onwards we let $G$ be a cyclic group of order $p$ and $\kk$ a field of characteristic $p$. Let $V$ and $W$ be $\kk G$-modules. We fix a generator $\sigma$ of $G$. Recall that, up to isomorphism, there are exactly $p$ indecomposable $\kk G$-modules $V_1, V_2, \ldots, V_p$, where the dimension of $V_i$ is $i$ and each has fixed-point space of dimension 1.  The isomorphism class of $V_i$ is usually represented by a module of column vectors on which $\sigma$ acts as left-multiplication by a single Jordan block of size $i$. 

Suppose $W \cong V_n$. It is convenient to choose a basis  $w_1, w_2, \ldots, w_{n}$ of $W$ for which the action of $G$ is given by

\begin{align*} \sigma w_1 &=w_1\\
\sigma w_2&=w_2-w_{1}\\
\sigma w_3 &= w_2-w_2+w_1\\
\vdots & \\
\sigma w_n &=w_n-w_{n-1}+w_{n-2}-\ldots \pm w_1. \end{align*}

(thus, the action of $\sigma^{-1}$ is given by left-multiplication by a upper-triangular Jordan block).
We do not (yet) choose a particular action on a basis for $V$, nor do we assume $V$ is indecomposable; we let $v_1,v_2, \ldots,v_m$ be a basis of $V$ and let $x_1, \ldots, x_m$ be the dual of this basis.

Note that $\kk[V] = \kk[x_1,x_2, \ldots, x_m]$, and a general element of $\kk[V, W]$ is given by
\[\phi = f_1w_1+f_2w_2+ \ldots+ f_nw_n\] where each $f_i \in \kk[V]$. We define the \textbf{support} of $\phi$ by
\[\supp(\phi) = \{i: f_i \neq 0\}.\]

The operator $\Delta = \sigma-1 \in \kk G$ will play a major role in our exposition. $\Delta$ is a \emph{$\sigma$-twisted derivation} on $\kk[V]$; that is, it satisfies the formula

\begin{equation}\label{derivation} \Delta(fg) = f\Delta(g)+\Delta(f)\sigma(g)\end{equation} for all $f, g \in \kk[V]$.  

Further, using induction and the fact that $\sigma$ and $\Delta$ commute, one can show $\Delta$ satisfies a Leibniz-type rule

\begin{equation}\label{Leibniz}
\Delta^k(fg) = \sum_{i=0}^k \begin{pmatrix} k\\i \end{pmatrix} \Delta^i( f) \sigma^{k-i}( \Delta^{k-i} (g)).
\end{equation}

A further result, which can be deduced from the above and proved by induction is the rule for differentiating powers:

\begin{equation}\label{powerrule}
\Delta(f^k) = \Delta(f)\left( \sum_{i=0}^{k-1} f^i \sigma(f)^{k-1-i} \right)
\end{equation}
for any $k \geq 1$.

For any $f \in \kk[V]$ we define the \textbf{weight} of $f$:
\[\wt(f) = \min\{i>0: \Delta^i(f) = 0\}.\]
Notice that $\Delta^{\wt(f)-1}(f) \in \ker(\Delta) = \kk[V]^G$ for all $f \in \kk[V]$. 
Another consequence of  (\ref{Leibniz}) is the following: let $f, g \in \kk[V]$ and set $d = \wt(f), e = \wt(g)$. Suppose that
\[d+e-1 \leq p.\] Then
\[\Delta^{d+e-1}(fg) = \sum_{i=0}^{d+e-1}\begin{pmatrix} d+e-1\\i \end{pmatrix}\Delta^i(f)\sigma^{d+e-1-i}(\Delta^{d+e-1-i}(g)) = 0\]
since if $i<e$ then $d+e-1-i>d-1$. On the other hand
\begin{align*}
\Delta^{d+e-2}(fg) &= \sum_{i=0}^{d+e-2}\begin{pmatrix} d+e-2\\i \end{pmatrix}\Delta^i(f)\sigma^{d+e-2-i}(\Delta^{d+e-2-i}(g)) \\
 &= \begin{pmatrix} d+e-2\\i \end{pmatrix}\Delta^{d-1}(f)\sigma^{e-1}(\Delta^{e-1}(g)) \neq 0 \end{align*}
since $\begin{pmatrix} d+e-2\\i \end{pmatrix} \neq 0 \mod p$. We obtain the followng:

\begin{Prop}\label{weightproduct} Let $f, g \in \kk[V]$ with $\wt(f)+\wt(g)-1 \leq p$. Then $\wt(fg) = \wt(f)+\wt(g)-1$.
\end{Prop}

Also note that
\[\Delta^p = \sigma^p-1=0\] which shows that $\wt(f) \leq p$ for all $f \in \kk[V]^G$. Finally notice that
\begin{equation}\label{trans} \Delta^{p-1} = \sum_{i=0}^{p-1} \sigma^i.\end{equation}
This is the {\it Transfer map}, a $\kk[V]^G$-homomorphism $\Tr^G: \kk[V] \rightarrow \kk[V]^G$ which is well-known to invariant theorists.  

Now we have a crucial observation concerning the action of $\sigma$ on $W$: for all $i=1, \ldots, n-1$ we have
\begin{equation}\label{crucial}
\Delta(w_{i+1})+\sigma(w_i)=0
\end{equation}
and $\Delta(w_1)=0$.

From this we obtain a simple characterisation of covariants:
\begin{Prop}\label{charcov} Let 
\[\phi =  f_1w_1+f_2w_2+ \ldots+ f_nw_n.\]

Then $\phi \in \kk[V,W]^G$ if and only if there exists $f \in \kk[V]$ with weight $\leq n$ such that $f_i = \Delta^{i-1}(f)$ for all $i=1, \ldots, n$.
\end{Prop}

\begin{proof} Assume $\phi \in \kk[V,W]^G$. Then we have
\begin{align*} 0 =& \Delta\left(\sum_{i=1}^n f_iw_i\right)\\	
				=& \sum_{i=1}^n \left(f_i\Delta(w_i)+\Delta(f_i)\sigma(w_i)\right)\\
				=&\sum_{i=1}^{n-1} (\Delta(f_i)-f_{i+1})\sigma(w_i) + \Delta(f_n)\sigma(w_n)
\end{align*}
where we used ($\ref{crucial}$) in the final step. Now note that $$\sigma(w_i)) = w_i + (\text{terms in}\ w_{i-1},w_{i-2}, \ldots, w_1)$$ for all $i=1, \ldots, n$. Thus, equating coefficients of $w_i$, for $i=n, \ldots, 1$ gives
\[\Delta(f_n)=0, \Delta(f_{n-1})=f_n, \ldots, \Delta(f_{2})=f_3, \Delta(f_1) = f_2.\]
Putting $f=f_1$ gives $f_i = \Delta^{i-1}(f)$ for all $i=1, \ldots, n$ and $0 = \Delta^n(f)$ as required.\\

Conversely, suppose that \[\phi = \sum_{i=1}^n \Delta^{i-1}(f) w_i\] for some $f \in \kk[V]$ with $\Delta^n(f)=0$. Then we have
\begin{align*} \Delta(\phi) &= \sum_{i=1}^n\Delta^{i-1} (f) \Delta(w_i) + \Delta^i(f) \sigma(w_i)\\
		&= \sum_{i=2}^n (-\Delta^{i-1} (f)  \sigma(w_{i-1}) + \Delta^i(f) \sigma(w_i)) + \Delta(f)\sigma(w_1) & \text{by (\ref{crucial})}\\
		&= \Delta^n(f) \sigma(w_n)\\
		&= 0 
\end{align*} as required.

\end{proof}

Note that the support of any covariant is therefore of the form $\{1,2, \ldots, i\}$ for some $i \leq n$. We will write
\[\supp(\phi) = i\] if $\phi$ is a covariant and $\supp(\phi) = \{1,2,\ldots, i\}$.

Introduce  notation
\[K_n:= \ker(\Delta^n)\]
and
\[I_n:= \im(\Delta^n).\]
These are $\kk[V]^G$-modules lying inside $\kk[V]$.

Now we can define a map

\[\Theta: K_n \rightarrow \kk[V,W]^G\]
\begin{equation}\label{theta} \Theta(f) = \sum_{i=1}^n \Delta^{i-1}(f)w_i.\end{equation}

Clearly $\Theta$ is an injective, degree-preserving map of $\kk[V]^G$-modules, and Proposition \ref{charcov} implies it is also surjective. We conclude that
\begin{Prop}\label{kniscov} $K_n$ and $\kk[V,W]^G$ are isomorphic as graded $\kk[V]^G$-modules.
\end{Prop}

From this point onwards we set $V=V_m$ and $W=V_n$, with the basis of $V$ chosen so that
\begin{align*} \sigma x_1 &= x_1+x_2,\\
\sigma x_2&=x_2+x_3,\\
\sigma x_3&= x_3+x_4.\\
\vdots\\
\sigma x_m &= x_m.\end{align*}

\begin{Lemma}\label{monomials} Let $z = x_1^{e_1}x_2^{e_2} \ldots x_m^{e_m}$. Let $d =  \sum_{i=1}^m e_i(m-i)$, $e = \sum_{i=1}^m e_i = \deg(z)$ and assume $d<p$. Then
$$\wt(z) = d+1.$$
\end{Lemma}

\begin{proof} Applying Proposition \ref{weightproduct} repeatedly and noting that $\wt(x_i) = m-i+1$, we find
\begin{align*}\wt(z) &= \sum_{i=1}^m\left(e_i (m-i+1) - e_i+1\right) - (n-1)\\
&= \sum_{i=1}^m (e_i(m-i))+1 = d+1.
\end{align*}
\end{proof}

\section{Hilbert series}\label{sec:hs}

Let $\kk$ be a field and let $S = \oplus_{i \geq 0} S_i$ be a positively graded $\kk$-vector space. The dimension of each graded component of $S$ is encoded in its Hilbert Series

\[H(S,t) = \sum_{i \geq 0} \dim(S_i) t^i. \]

Proposition \ref{kniscov} implies that $H(\kk[V,W]^G,t) = H(K_n,t)$. In this section we will outline a method for computing $H(K_n,t)$.

Each homogeneous component $\kk[V]_i$ of $\kk[V]$ is a $\kk G$-module. As such, each one decomposes as a direct sum of modules isomorphic to $V_k$ for some values of $k$. Write $\mu_k(\kk[V]_i)$ for the multiplicity of $V_k$ in $\kk[V]_i$ and define the series

\[H_k(\kk[V]) = \sum_{i \geq 0} \mu_k(\kk[V]_i) t^i.\]

The series $H_k(\kk[V_m])$ were studied by Hughes and Kemper in \cite{HughesKemper}. 
They can also be used to compute the Hilbert series of $\kk[V_m]^G$; since $\dim(V^G_k)=1$ for all $k=1, \ldots, p$ we have
\begin{equation}\label{hskvg} H(\kk[V_m]^G,t) =  \sum_{k=1}^p  H_k(\kk[V_m],t).\end{equation}

Now observe that 
\[\dim(\ker(\Delta^n|_{V_k})) = \left\{\begin{array}{lr} n & k \geq n\\
										    k & \text{otherwise.}	 \end{array}  \right.\]

Therefore
\[ H(K_n,t) = \sum_{k=1}^{n-1} k H_k(\kk[V],t) + \sum_{k=n}^p n H_k(\kk[V],t).\] 
We can write this as as a series not depending on $p$:
\begin{equation}\label{knseries} H(K_n,t) = nH(\kk[V]^G,t) - (\sum_{k=1}^{n-1} (n-k)H_k(\kk[V],t)).
\end{equation}
using equation (\ref{hskvg}).

We will need the Hilbert Series of $I^G_n = \kk[V]^G \cap I_n$ in the final section. For all $k=1, \ldots, p$ we have
\[\dim(\Delta^n(V_k))^G = \left\{\begin{array}{lr} 1 & k > n\\
										    0 & \text{otherwise.}	 \end{array}  \right.\]
Therefore \[ H(I_n^G,t) =  \sum_{k=n+1}^p H_k(\kk[V],t),\] which we can write independently of $p$ as
 \begin{equation}\label{ingseries} H(I^G_n,t) = H(\kk[V]^G,t) - (\sum_{k=1}^{n} H_k(\kk[V],t)).
\end{equation}

\section{Decomposition Theorems}\label{sec:decomp}

In this section we will  compute the series $H_k(\kk[V_2],t)$ and $H_k(\kk[V_3],t)$ for all $k=1, \ldots, p-1$. 

Hughes and Kemper \cite[Theorem~3.4]{HughesKemper} give the formula
\begin{equation}\label{hugheskemper}H_k(\kk[V_m],t) = \sum_{\gamma \in M_{2p}} \frac{\gamma-\gamma^{-1}}{2p}\gamma^{-k} \frac{1-\gamma^{p(m-1)}t^p}{1-t^p}\prod_{j=0}^{m-1}(1-\gamma^{m-1-2j}t)^{-1},\end{equation}

where $M_{2p}$ represents the set of $2p$th roots of unity in $\C$. A similar formula is given for $H_p(\kk[V],t)$ but we will not need this.
The following result can be derived from the formula above, but follows more easily from \cite[Proposition~3.4]{ElmerSymPowers}:

\begin{Lemma}\label{hkv2} $H_k(\kk[V_2,t]) = \frac{t^{k-1}}{1-t^p}$.
\end{Lemma}

For $V_3$ we will have to use Equation (\ref{hugheskemper}). This becomes
\[H_k(\kk[V_3],t) = \frac{1}{2p(1-t)} \sum_{\gamma \in M_{2p}} \frac{(\gamma-\gamma^{-1})\gamma^{-k+2}}{(1-\gamma^{2}t)(\gamma^2-t)}.\]

\begin{Lemma}\label{hkv3} $$H_k(\kk[V_3],t) = \left\{ \begin{array}{lr} \frac{t^{p-l}-t^{p-l-1}+t^{l+1}-t^{l}}{(1-t)(1-t^2)(1-t^p)} & \text{if}\ k = 2l+1\ \text{is odd} \\ 0 & \text{if}\ k\ \text{is even}. \end{array} \right.$$
\end{Lemma}

\begin{proof} We evaluate  \[\frac{(\gamma-\gamma^{-1})\gamma^{-k+2}}{(1-\gamma^{2}t)(\gamma^2-t)} = \frac{A}{\gamma - t^{\frac12}} + \frac{B}{\gamma + t^{\frac12}} + \frac{C}{1-\gamma t^{\frac12}} + \frac{D}{1+\gamma t^{\frac12}}\] using partial fractions, finding
\[A = \frac{t^{-l+1}-t^{-l}}{(2t^{\frac12})(1-t^2)},\]
\[B = (-1)^{-k+3}\frac{t^{-l+1}-t^{-l}}{(-2t^{\frac12})(1-t^2)},\]
\[C = \frac{t^{l-1}-t^{l}}{2(t^{-1}-t)},\]
\[D = (-1)^{-k+3}\frac{t^{l-1}-t^{l}}{2(t^{-1}-t)}.\]

Now we compute:
\begin{align*}\sum_{\gamma \in M_{2p}}\frac{1}{\gamma-t^{\frac12}} &= \sum_{\gamma \in M_{2p}}\frac{-t^{-\frac12}}{1-\gamma t^{-\frac12}}\\
&= -t^{-\frac12}\sum_{i=0}^{\infty} \sum_{\gamma \in M_{2p}} (\gamma t^{-\frac12})^i\\
&= -t^{-\frac12} 2p \sum_{i=0}^{\infty} (t^{-\frac12})^{2pi}\\
&=  -t^{-\frac12} 2p \frac{1}{1-(t^{-\frac12})^{2p}}\\
&= -t^{\frac12} 2p \frac{1}{1-t^{-p}}\\
&= 2p\frac{t^{p-\frac12}}{1-t^p}
\end{align*}
 Similarly we have
\[\sum_{\gamma \in M_{2p}} \frac{1}{\gamma+t^{\frac12}} =  -2p\frac{t^{p-\frac12}}{1-t^p}\]
while
\begin{align*}\sum_{\gamma \in M_{2p}} \frac{1}{1-\gamma t^{\frac12}}
&= \sum_{i=0}^{\infty} \sum_{\gamma \in M_{2p}}(\gamma t^{\frac12})^i\\
&= 2p \sum_{i=0}^{\infty} ( t^{\frac12})^{2pi}\\
&= 2p \sum_{i=0}^{\infty} ( t^{pi})\\
&= 2p \frac{1}{1-t^p}
\end{align*}
and similarly 
\[\sum_{\gamma \in M_{2p}} \frac{1}{1+\gamma t^{\frac12}} = 2p \frac{1}{1-t^p}\] as $\{-\gamma: \gamma \in M_{2p}\} = M_{2p}$.

It follows that 
\begin{align*} H_k(\kk[V_3],t) &= \frac{1}{2p(1-t)} \left(\frac{(A-B)2pt^{p-\frac12}}{1-t^p} + \frac{2p(C+D)}{1-t^p}\right) \\ 
&= \frac{1}{(1-t)(1-t^p)}\left( \frac{(1+(-1)^{-k+3})(t^{p-l}-t^{p-l-1})}{2(1-t^2)} + \frac{(1+(-1)^{-k+3})(t^{l-1}-t^l)}{2(t^{-1}-t)}\right)\\
&= \left\{ \begin{array}{lr} \frac{t^{p-l}-t^{p-l-1}+t^{l+1}-t^{l}}{(1-t)(1-t^2)(1-t^p)} & \text{if}\ k \ \text{is odd} \\ 0 & \text{if}\ k\ \text{is even} \end{array} \right.
\end{align*}
as required.
\end{proof}

\section{Main results: $V_2$}\label{sec:v2}

We are now in a position to state our main results. First, suppose $V = V_2$ and $W = V_n$ where $n \leq p$. Then it's well known that $\kk[V]^G$ is a polynomial ring, generated by $x_2$ and $$N = \prod_{i =0}^{p-1} \sigma^i(x_1) =  x_1^p-x_1x_2^{p-1}.$$
Therefore we have

\begin{equation}\label{hsv2} H(\kk[V]^G,t) = \frac{1}{(1-t)(1-t^p)}.
\end{equation}

\begin{Prop}\label{hsv2vn} We have $$H(K_n,t) = H(\kk[V,W]^G,t) = \frac{1+t+t^2+ \ldots+t^{n-1}}{(1-t)(1-t^p)}.$$
\end{Prop}

\begin{proof} Using equations (\ref{knseries}) and (\ref{hsv2}) and Lemma \ref{hkv2} we have
\[H(K_n,t) = \frac{n}{(1-t)(1-t^p)} - \sum_{k=1}^{n-1} \frac{(n-k)t^{k-1}}{1-t^p} = \frac{1+t+t^2+ \ldots+t^{n-1}}{(1-t)(1-t^p)}.\] 
The result now follows from Proposition \ref{kniscov}.
\end{proof}

\begin{Theorem}\label{mainv2}
The module of covariants $\kk[V,W]^G$ is generated freely over $\kk[V]^G$ by 
\[\{\Theta(x_1^k): k = 0, \ldots, n-1\}.\]
where $\Theta(x_1^0) = \Theta(1) = w_1$.
\end{Theorem}

Note that, by Proposition \ref{broerchuai}(i), $\kk[V,W]^G$ is free over $\kk[V]^G$ and we could use \cite[Theorem~3]{BroerChuaiRelative} to check our proposed module generators. However, we prefer a more direct approach.

\begin{proof} It follows from Lemma \ref{monomials} that $\wt(x_1^k) = k+1$. Therefore $\supp(\Theta(x_1^k)) = k+1$, and so it's clear that the $\kk[V]^G$-submodule $M$ of $\kk[V,W]^G$ generated by the proposed generating set is free. Moreover, as $\deg(\Theta(x_1^k)) = k$, $M$ has Hilbert series

\[\frac{1+t+t^2+ \ldots + t^{n-1}}{(1-t)(1-t^p)}.\] 

But by Proposition \ref{hsv2vn}, this is the Hilbert series of $\kk[V,W]^G$. Therefore $M = \kk[V,W]^G$ as required.
\end{proof}

\begin{Corollary}
$K_n$ is a free $\kk[V^G]$-module, generated by $\{x_1^k: k=0, \ldots, n-1\}$.
\end{Corollary}

\begin{proof}
Follows from Theorem \ref{mainv2} above and the proof of Proposition \ref{kniscov}.
\end{proof}

\begin{rem}
The above was also obtained, in the special case $n=p-1$, by Erku{\c s} and Madran \cite{ErkusMadran}.
\end{rem}

\section{Main results: $V_3$}\label{sec:v3}

In this section let $p$ be an odd prime and $V=V_3$. We begin by describing $\kk[V]^G$. This has been done in several places before, for example \cite{DicksonMadison} and \cite[Theorem~5.8]{Peskin}, but we include this for completeness.

We use a graded reverse lexicographic order on monomials $\kk[V]$ with $x_1>x_2>x_3$. If $f \in \kk[V]$ then the {\it lead term} of $f$ is the term with the largest monomial in our order and the {\it lead monomial} is the corresponding monomial. If $f,g \in \kk[V]$ we will write \[f>g\] if the lead monomial of $f$ is greater than the lead monomial of $g$.

The results of section \ref{sec:hs} can be used to show
\begin{equation}\label{hsv3} H(\kk[V]^G,t) = \frac{1+t^p}{(1-t)(1-t^2)(1-t^p)}.
\end{equation}

Note that using the given order, we have
\[f > \Delta(f)\] for all $f \in \kk[V]$.

We recall two popular means of constructing invariants. Let $f \in \kk[V]$. As mentioned in section \ref{sec:prelim}, the transfer
\[\Delta^{p-1}(f) = \Tr^G(f) = \sum_{i=0}^{p-1} (\sigma^i f) \]
and also the norm
\[N(f) = \prod_{i=0}^{p-1} (\sigma^i f) \] 
of $f$ both lie in $\kk[V]^G$. It is easily shown that
\begin{align*} a_1 &:= x_3,\\ a_2 &:= x_2^2-2x_1x_3-x_2x_3,\\ a_3 &:= N(x_1)=\prod_{i=0}^{p-1} \sigma^i(x_1) \end{align*}
are invariants, and looking at their lead terms tells us that they form a homogeneous system of parameters for $\kk[V]^G$, with degrees $1,2$ and $p$. 

\begin{Prop}\label{gensv3} Let $f \in \kk[V]^G$ be any invariant with lead term $x_2^p$. Let $A = \kk[a_1,a_2,a_3]$. Then $f \not \in A$. Consequently $\kk[V]^G$ is a free $A$-module, whose generators are $1$ and $f$.
\end{Prop}

\begin{proof} It is clear that $f \not \in A$, as its lead term is not in the subalgebra of $\kk[V]$ generated by the lead terms of $a_1,a_2$ and $a_3$. Therefore the $A$-submodule of $\kk[V]^G$ generated by $1$ and $f$ has Hilbert series
\[\frac{1+t^p}{(1-t)(1-t^2)(1-t^p)}\] which is the Hilbert series of $\kk[V]^G$ as required.
\end{proof}

The obvious choice of invariant with lead term $x_2^p$ is $N(x_2)$. However, we will use $\Tr^G(x_1^{p-1}x_2)$ instead. For the calculation of the lead term of this invariant see \cite[Lemma~3.1]{Shankv4v5} or Lemma \ref{monomialsx_1^ix_2} to come.

The following observation is a consequence of the generating set above.

\begin{Lemma}\label{obs} 
Let $f \in A$. Then the lead term of $f$ is of the form $x_1^{pi}x_2^{2j}x_3^k$ for some positive integers $i,j,k$.
\end{Lemma}

Now let $W = V_n$ for some $n \leq p$. For the rest of this section, we set $l = \frac12 n$ if $n$ is even, with $l = \frac12(n-1)$ if $n$ is odd. Our first task is to compute the Hilbert Series of $\kk[V,W]^G$. Once more we use equation (\ref{knseries}) and the bijection $\Theta$ to do this. We omit the details.

\begin{Prop}\label{hsv3vn}\
\[H(\kk[V,W]^G,t) = \frac{1+2t+2t^2+ \ldots +2t^l+2t^{p-l}+2t^{p-l+1}+\ldots + t^p}{(1-t)(1-t^2)(1-t^p)}\]
if $n$ is odd, while
\[H(\kk[V,W]^G,t) = \frac{1+2t+2t^2+ \ldots +2t^{l-1}+t^l+t^{p-l}+2t^{p-l+1}+\ldots + 2t^{p-1} + t^p}{(1-t)(1-t^2)(1-t^p)}\] if $n$ is even.
\end{Prop}

Next, we need some information about the lead monomials of certain polynomials:

\begin{Lemma}\label{monomialsx_1^ismall} Let $j \leq k < p$. Then $\Delta^j(x_1^k)$ has lead term $$\frac{k!}{(k-j)!}x_1^{k-j}x_2^j.$$
\end{Lemma}

\begin{proof} The proof is by induction on $j$, the case $j=0$ being clear. Suppose $1 \leq j<k$ and
\[\Delta^j(x_1^k) = \frac{k!}{(k-j)!}x_1^{k-j}x_2^j + g\] where $g \in \kk[V]$ has lead monomial $\leq x_1^{k-j-1}x_2^{j+1}$.
Then
\begin{align*}
\Delta^{j+1}(x_1^k) &= \frac{k!}{(k-j)!} \Delta(x_1^{k-j}x_2^j) + \Delta(g) \\
		&= \frac{k!}{(k-j)!} \Delta(x_1^{k-j}) \sigma(x_2^j) + x_1^{k-j} \Delta(x_2^j) + \Delta(g). \end{align*}

Note that the lead monomial of $\Delta(g)$ is $< x_1^{k-j-1}x_2^{j+1}$. Now applying (\ref{powerrule}) shows that $\Delta(x_2^j)$ is divisible by $x_3$ and
\begin{align*} \Delta(x_1^{k-j}) &= x_2(x_1^{k-j-1} + x_1^{k-j-2} \sigma(x_1) + \ldots + \sigma(x_1)^{k-j-1})\\
&= (k-j) x_1^{k-j-1}x_2 + \ \text{smaller terms}. \end{align*}

In addition, $$\sigma(x_2^j) = (x_2+x_3)^j = x_2^j + \ \text{smaller terms}.$$

Therefore the lead term of $\Delta^{j+1}(x_1^k)$ is
\[(k-j) \frac{k!}{(k-j)!}x_1^{k-j-1}x_2^{j+1} = \frac{k!}{(k-j-1)!} x_1^{k-j-1} x_2^{j+1}\] as required.

\end{proof}

Similarly we have
\begin{Lemma}\label{monomialsx_1^ix_2} Let $j \leq k < p$. Then $\Delta^j(x_1^kx_2)$ has lead term $$\frac{k!}{(k-j)!}x_1^{k-j}x_2^{j+1}.$$
\end{Lemma}

\begin{proof}
We have by (\ref{Leibniz}) \[\Delta^j(x_1^kx_2) = \sum_{i=0}^j \begin{pmatrix} j \\ i \end{pmatrix} \Delta^{j-i}(x_1^k)\sigma^i(\Delta^i(x_2)).\] Only the first two terms are nonzero, hence
\begin{align*}\Delta^j(x_1^kx_2) &= \Delta^j(x_1^k)x_2 + j \Delta^{j-1}(x_1^k)x_3.\\
&= \frac{k!}{(k-j)!}x_1^{k-j}x_2^{j+1} + \ \text{smaller terms} \end{align*}
where we used Lemma \ref{monomialsx_1^ismall} is the last step.
\end{proof}

We are now ready to state our main results. Let $V=V_3$ and $W=V_n$. For any $i = 0,1, \ldots, n-1$ we define monomials
\[M_i =  \left\{ \begin{array}{lr} x_1^{i/2} & \text{if $i$ is even,}\\ x_1^{(i-1)/2}x_2 & \text{if $i$ is odd}.\end{array} \right.\]
and polynomials 
\[P_i =  \left\{ \begin{array}{lr} \Delta(x_1^{p-i/2}) & \text{if $i$ is even, $i>0$,}\\ x_1^{p-(i+1)/2} & \text{if $i$ is odd}.\end{array} \right.\] with $P_0 = x_1^{p-1}x_2$.

\begin{Theorem}
\label{kngens}
Let $n \leq p$. Then $K_n$ is a free $A$-module, generated by 
\[S_n = \{M_0,M_1, \ldots, M_{n-1}, \Delta^{p-n}(P_0), \Delta^{p-n}(P_1), \ldots, \Delta^{p-n}(P_{n-1})\}.\]
\end{Theorem}

\begin{proof}
By Lemma \ref{weightproduct}, the weight of $M_i$ is $i+1$ for $i<p$, while the weight of $P_i$ is
\[\left\{\begin{array}{lr} p & i\ \text{odd or zero} \\ p-1 & i\ \text{even, $i>0$.} \end{array} \right.\] Therefore the given polynomials all lie in $K_n$. Further, the degree of $M_i$ is $\lceil \frac{i}{2} \rceil$ and the degree of $P_i$ is $p-\lceil \frac{i}{2} \rceil$ which shows that the $A$-module generated by $S_n$ has Hilbert series bounded above by the Hilbert series of $K_n$ given in Proposition \ref{hsv3vn}, with equality if and only if it is free. Therefore it is enough to prove that $S_n$ is linearly independent over $A$.

Applying Lemmas \ref{monomialsx_1^ismall} and \ref{monomialsx_1^ix_2}, the lead monomials of $S_n$ are
\[\{1,x_2, x_1, x_1x_2, \ldots, x_1^{l-1}x_2, x_1^l,\]\[ x_1^{n-l-1}x_2^{p-n+1}, x_1^{n-l}x_2^{p-n}, \ldots, x_1^{n-2}x_2^{p-n+1}, x_1^{n-1}x_2^{p-n}, x_1^{n-1}x_2^{p-n+1}\}\] if $n$ is odd, and 
\[\{1,x_2, x_1, x_1x_2, \ldots, x_1^{l-2}x_2, x_1^{l-1}, x_1^{l-1}x_2,\]\[ x_1^{n-l}x_2^{p-n}, x_1^{n-l}x_2^{p-n+1}, x_1^{n-l+1}x_2^{p-n} \ldots, x_1^{n-2}x_2^{p-n+1}, x_1^{n-1}x_2^{p-n}, x_1^{n-1}x_2^{p-n+1}\}\] if $n$ is even.

In either case, we note that none of the claimed generators have lead term divisible by $x_3$, that each has $x_1$-degree $<p$, that there are at most two elements in $S_n$ with the same $x_1$-degree, and that when this happens these elements have $x_2$-degrees differing by 1. Combined with Lemma \ref{obs}, we see that for every possible choice of $f \in A$ and $g \in S_n$, the lead monomial of $fg$ is different. Therefore there cannot be any $A$-linear relations between the elements of $S_n$.   
\end{proof}

\begin{rem}
A generating set for $K_{p-1}$ over a different system of parameters can be found in \cite{ErkusMadran}.
\end{rem}

\begin{Corollary}\label{mainv3}
Let $n \leq p$. Then $\kk[V,W]^G$ is a Cohen-Macaulay module, generated over $A$ by 
 \[\{\Theta(M_0),\Theta(M_1), \ldots, \Theta(M_{n-1}), \Theta(P_0), \Theta(\Delta^{p-n}(P_1)), \ldots, \Theta(\Delta^{p-n}(P_{n-1}))\}.\]
\end{Corollary}

\begin{proof}
Follows from Theorem \ref{kngens} and the proof of Proposition \ref{kniscov}.
\end{proof}

\section{Application to transfers}

The transfer ideal $\Tr^G(\kk[V])$ is widely studied in invariant theory. In the notation of this article, we have $\Tr^G(\kk[V]) = I^G_{p-1} = I_{p-1}$. In this section, we use our work on covariants to give minimal $\kk[V]^G$-generating sets of the the ideals $I^G_{n-1}$ for each $n=1,2, \ldots, p$ when $V=V_2$, and minimal $A$-generating sets of the the ideals $I^G_{n-1}$ for each $n=1,2, \ldots, p$ when $V=V_3$. We retain the notation of sections \ref{sec:v2} and \ref{sec:v3}.

\begin{Theorem}
Let $V=V_2$ and $1 \leq n \leq p$. Then $I^G_{n-1}$ is a free $\kk[V]^G$-module, generated by $x_2^{n-1}$.
\end{Theorem}

\begin{proof} The same argument as in Lemma \ref{monomialsx_1^ismall} implies that $\Delta^{n-1}(x_1^{n-1}) = \lambda x_2^{n-1}$ for some nonzero constant $\lambda$, so $x_2^{n-1} \in I^G_{n-1}$. Using (\ref{ingseries}) we see that
\[H(I^G_{n-1},t) = \frac{t^{n-1}}{(1-t)(1-t^n)}.\] As this is the Hilbert series of the ideal $x_2^{n-1}\kk[V]^G$, the result follows.
\end{proof}

For $V=V_3$ we need to do a bit more work. We define a set of invariants
\[T_{n-1} = \{\Delta^{n-1}(M_{n-1})\} \cup \{ \Delta^{p-1}(P_i): i\ \text{odd or zero}, i < n\}.\]
Bearing in mind the weight of $M_{n-1}$ is $n$, and the weight of each $P_i$ above is $p$, it's clear that $T_{n-1} \subset I^G_{n-1}$. We claim that
\begin{Prop}
$T_{n-1}$ generates $I^G_{n-1}$ as an $A$-module.
\end{Prop}

\begin{proof}
Let $h \in I^G_{n-1}$. Then we can write $h = \Delta^{n-1}(f)$ for some $f \in \kk[V]^G$ with weight $n$, and by Proposition \ref{charcov} we have $\Theta(f) \in \kk[V,V_{n}]^G$. By Corollary \ref{mainv3} we can find elements $\alpha_0,\alpha_1, \ldots, \alpha_{n-1}, \beta_0, \beta_1, \ldots, \beta_{n-1} \in A$ such that
\[\Theta(f) = \sum_{i=0}^{n-1} \alpha_i \Theta(M_i) + \sum_{i=0}^{n-1} \beta_i \Theta(\Delta^{p-n}(P_i)).\]
Equating coefficients of $w_{n}$ in the above we obtain
\[h = \sum_{i=0}^{n-1} \alpha_i \Delta^{n-1}(M_i) + \sum_{i=0}^{n-1} \beta_i \Delta^{p-1}(P_i))\] but since $\Delta^{n-1}(M_i)=0$ for $i<n-1$ and $\Delta^{p-1}(P_i)=0$ when $i$ is even and $i>0$, we get $h \in AT_n$ as desired.
\end{proof}

$T_{n-1}$ does not generate $I^G_{n-1}$ freely over $A$. To see this, note that if $T_{n-1}$ were free over $A$, the resulting module would have Hilbert series

\[\frac{t^{l}+t^{p-l}+t^{p-l+1}+ \ldots+t^p}{(1-t)(1-t^2)(1-t^p)}.\]
But using (\ref{ingseries}) to calculate the Hilbert series of $I^G_n$ yields
\begin{equation}\label{hsingv3}
H(I^G_{n-1},t) = \frac{t^l+t^{p-l}}{(1-t)(1-t^2)(1-t^p)}
\end{equation}
which is strictly smaller. We claim, however, that $T_n$ is a minimal generating set. The first step in our argument requires more knowledge of certain lead monomials:

\begin{Lemma}\label{monomialsx_1^ilarge} Let $j \leq k$ with $j+k<p$. Then $\Delta^{k+j}(x_1^k)$ can be expressed as

$$2^{-j} (j+k)! \begin{pmatrix} k \\ j \end{pmatrix} x_2^{k-j}x_3^j + \mu_{j,k} x_1 x_2^{k-j-2} x_3^{j+1} + \ \text{smaller terms}$$ for some constant $\mu_{j,k} \in \kk$, where $\mu_{j,k} = 0$ if $j-k < 2$. In particular, the lead monomial of $\Delta^{k+j}(x_1^k)$ is $x_2^{k-j}x_3^j$.
\end{Lemma}

\begin{proof} For shorthand we write
$$\lambda_{j,k} = 2^{-j}(j+k)! \begin{pmatrix} k \\ j \end{pmatrix}.$$
We begin by showing, for all $0 < j \leq k$, that
\begin{equation}\label{lambdarelation} \lambda_{j,k+1} = (j+k+1)\lambda_{j,k} + \begin{pmatrix} j+k+1\\ 2 \end{pmatrix} \lambda_{j-1,k}.
\end{equation}

The author wishes to thank Fedor Petrov for pointing out this fact. To prove it, note that 
\begin{align*}
\begin{pmatrix} j+k+1 \\ 2 \end{pmatrix}\lambda_{j-1,k} + (j+k+1)\lambda_{j,k} \end{align*}
\begin{align*}
&= \frac{(j+k+1)(j+k)}{2} 2^{-j+1} (j+k-1)! \begin{pmatrix} k \\ j-1\end{pmatrix} + (j+k+1) 2^{-j} (j+k)! \begin{pmatrix} k \\ j \end{pmatrix} \\
&= 2^{-j}(j+k+1)!\left( \begin{pmatrix} k \\ j-1 \end{pmatrix} + \begin{pmatrix} k \\ j \end{pmatrix} \right)\\
&= 2^{-j}(j+k+1)! \begin{pmatrix} k+1 \\ j \end{pmatrix}\\
&= \lambda_{j,k+1} 
\end{align*}
as required.

The proof is by induction on $j$. First suppose $j=0$. We must show that
\begin{equation}\label{j=0} \Delta^k(x_1^k) = k!x_2^k + \mu_{0,k} x_1 x_2^{k-2} x_3 + \ \text{smaller terms}.\end{equation} We prove this by induction on $k$. The case $k=1$ is clear (with $\mu_{0,1} = 0$), so let $k \geq 1$. Then we have 

\begin{align*} \Delta^{k+1}(x_1^{k+1}) &= \Delta^{k+1} (x_1^k \cdot x_1) \\
 &= \sum_{i=0}^{k+1} \begin{pmatrix} k+1 \\ i \end{pmatrix} \Delta^{k+1-i}(x^k_1) \sigma^{i}( \Delta^{i}(x_1))\\
&= x_1 \Delta^{k+1}(x_1^k) + (k+1) (x_2+x_3) \Delta^k(x_1^k) + \begin{pmatrix} k+1 \\ 2 \end{pmatrix}x_3 \Delta^{k-1}(x_1^k).
\end{align*}

Now by Lemma \ref{monomialsx_1^ismall} we have
\[ \Delta^{k-1} (x_1^k) = k! x_1 x_2^{k-1} + f\] for some $f \in \kk[V]$ with lead monomial $\leq x_2^k$. By induction we have
\[\Delta^k(x_1^k) = k! x_2^k + \mu_{0,k} x_1x_2^{k-2} x_3 + \ \text{smaller terms} \] and
\[\Delta^{k+1}x_1^k = k! \Delta(x_2^k) + \mu_{0,k} x_3 \Delta(x_1x_2^{k-2}) + \text{smaller terms}.\] 
\[= k! x_3(x_2^{k-1}+x_2^{k-2}\sigma(x_2) + \ldots + \sigma(x_2)^{k-1}) + \mu_{0,k} x_3 (x_2 \sigma(x_2^{k-2}) + x_1 \Delta(x^{k-2})) +  \ \text{smaller terms} \]
\[= (k. k! + \mu_{0,k}) x_2^{k-1}x_3+ \ \text{smaller terms}. \]
 So, ignoring terms smaller than $x_1x_2^{k-1}x_3$ we have

\begin{align*} \Delta^{k+1}(x_1^{k+1}) &= (k.k!+\mu_{0,k}) x_1x_2^{k-1}x_3+(k+1)!x_2^{k+1}+ (k+1)\mu_{0,k}x_1x_2^{k-1}x_3+k!\begin{pmatrix} k+1 \\ 2 \end{pmatrix}x_1x_2^{k-1}x_3\\
&= (k+1)!x_2^{k+1} + (k!(k+ \begin{pmatrix} k+1 \\ 2 \end{pmatrix})+ (k+2) \mu_{0,k})x_1x_2^{k-1}x_3
\end{align*}
from which the claim (\ref{j=0}) follows.

Now suppose $j>0$. We proceed by induction on $k$. The initial case is $k=j$, so we must first show that
\[\Delta^{2k}(x_1^k) = 2^{-k}(2k)! x_3^k.\]
We prove this by induction on $k$. The result is clear when $k=1$. Suppose that $k\geq 1$, then we have  by (\ref{Leibniz})
\[\Delta^{2k+2}(x_1^{k+1}) = x_1 \Delta^{2k+2}(x^k_1) + (2k+2)(x_2+x_3)\Delta^{2k+1}(x_1^k) + \frac{(2k+2)(2k+1)}{2}x_3\Delta^{2k}(x_1^k).\]
But by Lemma \ref{monomials}, the weight of $x_1^k$ is $2k+1$, so the first two terms vanish. By induction we are left with
\[\Delta^{2k+2}(x_1^{k+1}) = \frac{(2k+2)(2k+1)}{2} x_3 \frac{(2k)!}{2^k}x_3^k = \frac{(2k+2)!}{2^{k+1}}x_3^{k+1}\] as required. 

 Now suppose $k \geq j$, then we have

\begin{align*} \Delta^{j+k+1}(x_1^{k+1}) &= \Delta^{j+k+1}(x_1^{k} \cdot x_1) \\
 &= \sum_{i=0}^{j+k+1} \begin{pmatrix} j+k+1 \\ i \end{pmatrix} \Delta^{j+k+1-i}(x^{k}_1) \sigma^{i}( \Delta^{i}(x_1))\\
&= x_1 \Delta^{j+k+1}(x_1^{k}) + (j+k+1) (x_2+x_3) \Delta^{j+k}(x_1^{k}) + \begin{pmatrix} j+k+1 \\ 2 \end{pmatrix}x_3 \Delta^{j-1+k}(x_1^{k}).
\end{align*}

Now by induction on $k$ we have we have
\[ \Delta^{j+k}(x_1^k) = \lambda_{j,k} x_2^{k-j}x_3^j + \mu_{j,k} x_1 x_2^{k-j-2} x_3^{j+1} + \ \text{smaller terms}.\]
So
\begin{align*} \Delta^{j+k+1}(x_1^k) &= \lambda_{j,k} x_3^j \Delta(x_2^{k-j}) + \mu_{j,k} x_3^{j+1} \Delta(x_1x_2^{k-j-2}) +  \ \text{smaller terms}\\\\
&= \lambda_{j,k}x_3^j(x_3)(x_2^{k-j-1}+x_2^{k-j-2}\sigma(x_2)+ \ldots + \sigma(x_2)^{k-j-1}) \\&
+ \mu_{j,k}x_3^{j+1} (x_2\sigma(x_2^{k-j-2}) + x_1\Delta(x_2^{k-j-2})) + \ \text{smaller terms}\\\\
&= (\lambda_{j,k}(k-j) + \mu_{j,k}) x_3^{j+1}x_2^{k-j-2} + \ \text{smaller terms}. \end{align*}
Also by induction on $j$ we have
\[\Delta^{j-1+k}(x_1^k) = \lambda_{j-1,k} x_2^{k-j+1}x_3^{j-1}+ \mu_{j-1,k} x_1 x_2^{k-j-1}x_3^j + \ \text{smaller terms}.\]

So, ignoring terms smaller than $x_1x_2^{k-j-1}x_3^{j+1}$ we have
\begin{align*} \Delta^{j+k+1}(x_1^{k+1}) &=   (\lambda_{j,k}(k-j) + \mu_{j,k})x_1 x_3^{j+1}x_2^{k-j-2}\\&
+ (j+k+1)(\lambda_{j,k} x_2^{k+1-j}x_3^j + \mu_{j,k} x_1 x_2^{k-j-1} x_3^{j+1})\\&
 + \begin{pmatrix} j+k+1 \\ 2 \end{pmatrix}(\lambda_{j-1,k} x_2^{k-j+1}x_3^{j}+ \mu_{j-1,k} x_1 x_2^{k-j-1}x_3^{j+1})\\\\
=&\left((j+k+1) \lambda_{j,k} + \begin{pmatrix} j+k+1 \\ 2 \end{pmatrix} \lambda_{j-1,k}\right) x_2^{k+1-j}x_3^j \\+& (\lambda_{j,k}(k-j)+ (j+k+2)\mu_{j,k}+ \begin{pmatrix} j+k+1 \\& 2 \end{pmatrix}\mu_{j-1,k})x_1x_2^{k-j-1}x_3^{j+1}\\\\
&= \lambda_{j,k+1} x_2^{k+1-j}x_3^j + \\&(\lambda_{j,k}(k-j)+ (j+k+2)\mu_{j,k}+ \begin{pmatrix} j+k+1 \\ 2 \end{pmatrix}\mu_{j-1,k})x_1x_2^{k-j-1}x_3^{j+1}
\end{align*}
where we used the observation at the beginning of the proof in the final step.

This completes the proof of the formula for $\Delta^{j+k}(x_1^k)$. Finally, note that $\lambda_{j,k} \neq 0$ modulo $p$ if $j+k<p$.
\end{proof}

We can use this result, along with Lemma \ref{monomialsx_1^ix_2} to determine the lead monomial of each element of $T_{n-1}$: we have

\begin{itemize}
    \item $LM(\Delta^{n-1}M_{n-1}) = x_3^l$;
    \item $LM(\Delta^{p-1}(P_0)) = x_2^p$;
    \item $LM(\Delta^{p-1}(P_i)) = x_2^{p-i}x_3^{(i-1)/2}$ when $i$ is odd.
\end{itemize}

In particular for each $i<n$ odd or $i=0$ we have that 
$$\Delta^{p-1}(P_i) \not \in A(\Delta^{n-1}(M_{n-1}),\Delta^{p-1}(P_{j}): j>i, \text{j odd}),$$ which is the the ideal generated by the elements of $T_{n-1}$ with degree smaller than the degree of $\Delta^{p-1}(P_i)$, since each of these had lead monomial divisible by a larger power of $x_3$ than $(i-1)/2$. This shows that $T_{n-1}$ is indeed a minimal generating set.

\bibliographystyle{plain}
\bibliography{MyBib}

\end{document}